\documentclass[10pt,a4paper,oneside]{amsproc}

\usepackage{latexsym, amsthm}
\usepackage{amsfonts, amsmath, amssymb,comment}
\usepackage{hyperref}
\usepackage{euscript,mathrsfs}
\usepackage{url}
\usepackage{array}
\usepackage[a4paper]{geometry}
\usepackage{graphics}
\usepackage[usenames]{color}
\usepackage[all]{xy}
\usepackage{graphicx}
\usepackage{boxedminipage}
\usepackage{enumerate}

\newtheorem{theorem}{Theorem}[section]

\newtheorem{conjecture}{Conjecture}
\newtheorem{corollary}[theorem]{Corollary}

\newtheorem{definition}[theorem]{Definition}
\newtheorem{remark}[theorem]{Remark}

\newtheorem{lemma}[theorem]{Lemma}

\newtheorem{proposition}[theorem]{Proposition}

\newcommand{\codim}{\mathrm{codim}}
\newcommand{\ev}{\mathrm{ev}}
\newcommand{\wt}{\mathrm{wt}}
\newcommand{\B}{\mathcal{B}}
\newcommand{\s}{\mathcal{S}}
\newcommand{\Fqt}{\mathbb{F}_{q^2}}
\newcommand{\Fq}{\mathbb{F}_q}

\newcommand{\C}{\mathcal{C}}
\newcommand{\I}{\mathcal{I}}
\newcommand{\U}{\mathcal{U}}
\newcommand{\T}{\mathcal{T}}

\def\Fq{{\mathbb F}_q}
\def\AA{{\mathbb A}}
\def\FF{{\mathbb F}}
\def\PP{{\mathbb P}}

\newcommand{\X}{\mathcal{X}}
\newcommand{\p}{\mathsf{P}}

\begin{document}

\title[]{Functional codes arising from rank $n$ Hermitian varieties and  hypersurfaces in low dimensions}

\author{Subrata Manna}
\address{Department of Mathematics  \newline \indent
Indian Institute of Technology Bombay, Mumbai, India}
%\curraddr{}
\email{subrata147314@gmail.com}

\keywords{Functional codes, Hermitian varieties, Hypersurfaces, Rational points}
\subjclass[2020]{Primary 14G50, 14G05, 05B25}

\begin{abstract}   
We study the functional code $C_d(\X)$, introduced by G. Lachaud in 1996, in the case where $\X$ is a rank $n$ degenerate Hermitian variety $P\U_{n-1}$ in $\PP^n(\FF_{q^2})$ and $d\leq q$. We establish an upper bound for the maximum number of $\FF_{q^2}$-rational points in the intersection of $P\U_{n-1}$ with an $\Fqt$-hypersurface of degree at most $q$ in $\PP^n$. Using this bound, we determine the parameters of the codes $C_d(P\U_{n-1})$ in the cases $n=2,3,4$. We also characterize the hypersurfaces that correspond to the minimum distance of these codes in the cases $n=2,3,4$.
\end{abstract}

\date{}
\maketitle

%%%%%%%%%%%%%%%%%%%%%%%%%%%%%%%%%%%%%%%%%%%%%%%%%%%%%%%%%%%%%%%%%%%%%%%%%%%%%%%%%%%%%%%%%%%%%%%%%%%%%%%%%%%%%%%%%%%%%%%%%%%%%%%%%%%%%%%%%%%%%%%%%%%%%%%%%
\section{Introduction}\label{intro}
In recent years, functional codes associated with algebraic varieties over finite fields have attracted significant attention from algebraic geometers, coding theorists and combinatorists. This type of code was first introduced by G. Lachaud in \cite{L}. Since their introduction in $1996$, functional codes have been extensively studied in numerous articles; see, for instance, \cite{BD, BDH, E, E1, E2, EHRS, EHRS1}. For a survey of functional codes associated with several classes of varieties, we refer to \cite{L08}. Among the many classes of varieties studied in this context, Hermitian varieties, one of the celebrated objects in finite geometry, have received particular attention. 

We fix an integer $q$, a power of a prime integer. As customary, $\Fq$ and $\Fqt$ denote the finite fields with $q$ and $q^2$ elements, respectively. For $n \ge 0$, we denote by $\PP^n$ (resp. $\AA^n$) the projective space (resp. affine space) of dimension $n$ over the algebraic closure $\overline{\FF}_q$, while $\PP^n (\Fqt)$ (resp. $\AA^n(\Fqt)$) will denote the set of all $\Fqt$-rational points on $\PP^n$ (resp. $\AA^n$). By an algebraic variety, we mean a set of zeroes of a certain collection of polynomials in the affine space or projective space, depending on the context. In particular, an algebraic variety need not be irreducible in Zariski topology.

Let $\X$ be a projective algebraic variety over the finite field $\Fq$. The functional codes $C_d(\X)$ are defined by evaluating homogeneous polynomials of degree $d$ over rational points of $\X$. The parameters and the weight distribution of functional codes for $d = 2$, where $\mathcal{X}$ is a non-degenerate Hermitian variety in $\mathbb{P}^n$, have been studied by many authors; see, for example, \cite{BBFS, E, E1, E2, E3, ELX, HS}. The case when $d=3$ and $\X$ is a non-degenerate Hermitian variety has been studied in \cite{DM, S1}. Edoukou, Hallez, Rodier and Storme studied functional codes in the case when $d=2$ and $\X$ is a nonsingular quadric in \cite{EHRS1}. The functional codes $C_{Herm}(\X)$ over $\Fq$ are defined analogously by evaluating the Hermitian forms of degree $q+1$ over the rational points of $\X$. When $\X$ is a non-degenerate Hermitian variety $\U_n$ in $\PP^n(\Fqt)$, the code $C_{Herm}(\U_n)$ was studied in \cite{EHRS}, while in the case where $\X$ is a non-singular quadric quadric $\mathcal{Q}$, the code $C_{Herm}(\mathcal{Q})$ was studied in \cite{BBFS}.

The aim of this paper is to investigate the structure of the functional codes 
$C_d(P\U_{n-1})$, where $P\U_{n-1}$ is a rank $n$ Hermitian variety in $\PP^n$. 
We determine the length, dimension, and a lower bound for the minimum distance 
of $C_d(P\U_{n-1})$. In addition, for the cases $n=2, 3, 4$, we determine 
the code parameters and describe the hypersurfaces that give rise to 
minimum-weight codewords.

The rest of the paper is organized as follows. In Section \ref{prel}, we recall several well-known properties of Hermitian varieties, review some preliminary results from algebraic geometry, present bounds on the number of rational points contained in a variety, and briefly introduce the notions of linear codes. In Section \ref{degenerate of rank n}, we study the maximum size of the intersection of degenerate Hermitian varieties of rank $n$ with hypersurfaces and characterize the hypersurfaces attaining this maximum in the case when $n=2, 3,4$. Finally, in Section \ref{parameters}, we determine the parameters of the functional codes associated with rank $n$ Hermitian varieties and hypersurfaces.

\section{Preliminaries}\label{prel}

Throughout this paper, for a homogeneous polynomial $F \in \Fqt[x_0, \dots , x_n]$, we denote by $V(F)$, the set of zeroes of $F$ in $\PP^n$ and by $V(F)(\Fqt)$ the set of all $\Fqt$-rational points of $V(F)$. When we say that a hypersurface $V(F)$ is irreducible, we will mean that the polynomial $F$ is irreducible in the field of its definition. Furthermore, by an $\Fqt$-variety $\X$, we mean that the defining polynomials of $\X$ have coefficients in $\Fqt$.
All the results presented in this section are already well established, and their proofs can be found in the cited references.

%%%%%%%%%%%%%%%%%%%%%%%%%%%%%%%%%%%%%%%%%%%%%%

\subsection{Hermitian varieties over a finite field}\label{Geometry of Hermitian}
First, we revisit the definition and key properties of Hermitian varieties (cf. \cite{BC, C}), which will be essential in the later sections of this paper.

\begin{definition}[\cite{BC}] \normalfont
Let $H = (h_{ij})$ ($0 \le i, j \le n$) be an $(n+1) \times (n+1)$ matrix with entries from $\Fqt$. We denote by $H^{(q)}$ the matrix whose $(i, j)$-th entry is given by $h_{ij}^q$. The matrix $H$ is said to be a \textit{Hermitian matrix} if $H \neq 0$ and $H^T = H^{(q)}$, i.e., $h_{ij}=h_{ji}^q$ for all $i,j$. A \textit{Hermitian variety} of dimension $n-1$ is the zero set of the polynomial $x^T H x^{(q)}$ inside $\PP^n$, where $H$ is an $(n+1) \times (n+1)$ Hermitian matrix and $x = (x_0, \dots, x_n)^T$. The Hermitian variety is classified as \textit{non-degenerate} if $\mathrm{rank} \ H = n+1$ and \textit{degenerate} otherwise.
\end{definition}

Bose and Chakravarti in \cite[Corollary of Theorem 4.1]{BC} proved that if the rank of a Hermitian matrix is $r$, then after a suitable projective change of coordinates over $\Fqt$, we can recognize the corresponding Hermitian variety as the zero set of the polynomial equation
\begin{equation}\label{reduceherm}
x_0^{q+1} + x_1^{q+1} + \dots + x_{r-1}^{q+1} = 0.
\end{equation}
It is worth noting that the polynomial $x_0^{q+1} + x_1^{q+1} + \dots + x_{r-1}^{q+1}$ is irreducible over $\overline{\FF}_q$ if and only if $r\geq 3$. Thus, it follows that a Hermitian variety of rank at least three is always absolutely irreducible. 

From now onwards, a Hermitian variety in $\PP^n$ of rank $r$ will be denoted by $\U_{r,n}$. But whenever the variety is non-degenerate, that is of rank $n+1$, we will simply denote it by $\U_n$. It follows from \eqref{reduceherm} that when $r=n$, the degenerate Hermitian variety $\U_{n,n}$ is defined by the polynomial 
\begin{equation}\label{nreduce}
 x_0^{q+1}+x_1^{q+1}+\dots+x_{n-1}^{q+1}.   
\end{equation}
From geometric point of view, the zeros of \eqref{nreduce} represents a cone with vertex at the point $[0:\dots:0:1]\in \PP^n$ and base the non-degenerate Hermitian variety $\U_{n-1}=V\left(x_0^{q+1}+\dots+x_{n-1}^{q+1}\right)$ contained in the hyperplane $V(x_n)\cong \PP^{n-1}$. Consequently, a degenerate Hermitian variety $\U_{n,n}$ of rank $n$ is a cone over a non-degenerate Hermitian variety $\U_{n-1}$ with vertex at an $\Fqt$-point $P$ and hence we denote it by $P\U_{n-1}$.

The number of $\Fqt$-rational points of a non-degenerate Hermitian variety $\U_n$ is also well known and well established in the literature.
\begin{theorem}\cite[Theorem 8.1]{BC}\label{nondeg point}\ Let $\U_n$ be a non-degenerate Hermitian variety in $\PP^n$. Then $$|\U_n(\Fqt)|=\frac{\big(q^n-(-1)^n\big)\left(q^{n+1}-(-1)^{n+1}\right)}{q^2-1}.$$
\end{theorem}

\noindent In particular, the number of $\Fqt$-rational points in $P\U_{n-1}$ is $1+q^2|\U_{n-1}(\Fqt)|$. The hyperplane sections of a non-degenerate Hermitian variety are well-studied, thanks to \cite{BC, C}.

\begin{theorem}\label{hyperplane section}
     Let $\U_n$ be a non-degenerate Hermitian variety in $\PP^n$ and $\Sigma$ be a hyperplane of $\PP^{n}$ defined over $\Fqt$. 
     \begin{enumerate}
     \item[(a)] \cite[Theorem 7.4]{BC} if $\Sigma$ is tangent to $\U_n$ at a point $R\in \U_n(\Fqt)$, then $\Sigma \cap \U_n$ is a degenerate Hermitian variety of rank $n-1$ contained in $\Sigma$. That is, 
     $\Sigma \cap \U_n$ is a cone over a non-degenerate Hermitian variety $\U_{n-2}$ contained in a hyperplane of $\Sigma$ with center at $R$. 
     \item[(b)] \cite[Theorem 3.1]{C} if $\Sigma$ is not a tangent to $\U_n$, then $\Sigma \cap \U_n$ is a non-degenerate Hermitian variety $\U_{n-1}$ in $\Sigma$.
     \end{enumerate}
     
 \end{theorem} 
As we will be concerned with degenerate Hermitian varieties, we recall the subspace sections of such varieties.

\begin{theorem}\cite[Theorem 7.1]{BC}\label{disjoint from sing sp}
   Let $\U_{r,n}$ be a degenerate Hermitian variety of rank $r<n+1$ in $\PP^n$ and let $\Pi_{r-1}$ be an $\Fqt$-linear subspace of $\PP^n$ of dimension $r-1$ disjoint with singular space $\Pi_{n-r}$ of $\U_{r,n}$. Then $\U_{r,n}$ and $\Pi_{r-1}$ intersect in a non-degenerate Hermitian variety $\U_{r-1}$ in $\Pi_{r-1}$.
\end{theorem}

Since the singular space of a rank $n$ Hermitian variety $P\U_{n-1}$ is $\{P\}$, the corollary below follows immediately from Theorem \ref{disjoint from sing sp}.
\begin{corollary}\label{hyp cup}
      A degenerate Hermitian variety $P\U_{n-1}$ of rank $n$ in $\PP^n$ and an $\Fqt$-hyperplane $\Sigma$ with $P\notin \Sigma$ intersect in a non-degenerate Hermitian variety $\U_{n-1}$ in $\Sigma\cong \PP^{n-1}$.
\end{corollary}

\noindent We now conclude this subsection by describing the possibilities that arise when a line intersects a Hermitian variety. 

\begin{lemma}\cite[Section 7]{BC}\label{line}
Any line in $\PP^n (\Fqt)$ satisfies precisely one of the following:
\begin{enumerate}
\item[(a)] The line intersects $\U_{n}$ at precisely $1$ point. Such type of lines are called \textit{tangent lines}.
\item[(b)] The line intersects $\U_n$ at exactly $q+1$ points. Such type of lines are referred to as \textit{secant lines}. 
\item[(c)] The line is contained in $\U_n$.
\end{enumerate}
\end{lemma}
%%%%%%%%%%%%%%%%%%%%%%%%%%%%%%%%%%%%%%%%%%%%%%%

%%%%%%%%%%%%%%%%%%%%%%%%%%%%%%%%%%%%%%%%%%%%%%

\subsection{Preliminaries from algebraic geometry}\label{Prel from alg}

Here, we revisit several fundamental results from algebraic geometry and some well-known bounds on the number of rational points in a variety, which will be instrumental in the later part of the paper. We will rely on the notions of \textit{dimension} and \textit{degree} of a variety, as outlined in standard textbooks of Algebraic Geometry, such as Harris' book \cite{H}. We begin with the following definitions. A variety is said to be of \textit{equidimensional} if all its irreducible components have equal dimension. Additionally, two equidimensional varieties $X,Y\subset \PP^n$ are said to \textit{intersect properly} if $\codim(X\cap Y)=\codim X+\codim Y$.  We have the following proposition regarding the intersection of two hypersurfaces.

\begin{proposition}\cite[Proposition 2.13]{S1}\label{coprime}
Let $F, G \in \Fqt[x_0, x_1, \dots, x_n]$ be non-constant homogeneous polynomials having no common factors. Then 
\begin{enumerate}
\item[(a)] $V(F)$ and $V(G)$ intersect properly. 
\item[(b)] $V(F) \cap V(G)$ is equidimensional of dimension $n-2$.
\item[(c)] $\deg \left(V(F) \cap V(G)\right) \le \deg F \deg G$. 
\end{enumerate}
\end{proposition}

%%%%%%%%%%%%%%%%%%%%%%%%%%%%%%%%%%%%%%%%%%%%%%%%%%

%%%%%%%%%%%%%%%%%%%%%%%%%%%%%%%%%%%%%%%%%%%%%%%%%%
We now state a result due to J. P. Serre \cite{S}, proved independently by S{\o}rensen in \cite{So} on the number of rational points in a hypersurface. We refer \cite{GL} for a recent proof.
\begin{theorem}[Serre's inequality]\label{serre}
Let $ F \in \Fq[x_0, x_1, \dots, x_n]$ be a non-zero homogeneous polynomial of degree $d \le q$. Then 
$$|V(F)(\FF_q)| \le dq^{n-1} + \pi_{n-2}(\Fq),$$
where $\pi_{n-2}(\Fq) = 1 + q + \dots + q^{n-2}$. Moreover, the
equality holds if and only if $V(F)$ is a union of $d$ hyperplanes defined over $\Fq$, all containing a common linear subspace of codimension $2$. 
\end{theorem}

 In particular, any plane curve defined over $\Fqt$ of degree $d \leq q$ has at most $dq^2+1$ many $\Fqt$-rational points. Furthermore, a plane curve of degree $d\leq q$ that admits exactly $dq^2+1$ many $\Fqt$-rational points is a union of $d$ lines passing through a common $\Fqt$-point. We will also use bounds on the number of rational points of a variety that is not necessarily a hypersurface. To this end, we recall a result due to Lachaud and Rolland. For some recent proof, we refer to \cite{DG1}.

\begin{proposition}\cite[Prop. 2.3]{LR}\label{lac}
Let $X$ be an equidimensional projective (resp. affine) variety defined over a finite field $\Fq$. Further assume that $\dim X = \delta$ and $\deg X = d$.  Then
$$|X(\Fq)| \le d \pi_{\delta}(\Fq) \ \ \ \ \ \ (\mathrm{resp.} \ \  |X (\Fq)| \le d q^{\delta}),$$
where $\pi_{\delta}(\Fq) = 1 + q + \dots + q^{\delta}$.
\end{proposition}

Now we state a more recent result on the structure of a surface in $\PP^3$ that will be immensely useful in what follows.

\begin{theorem}\label{Datta}\cite[Theorem 3.1]{D}
    Let $Y \subset \PP^3$ be a surface of degree $d$ defined over $\Fq$ and $P \in Y(\Fq)$. Then one of the following holds: 
    \begin{enumerate}
        \item[(a)] $Y$ contains a plane defined over $\Fq$, 
        \item[(b)] $Y$ contains a cone over a plane curve defined over $\Fq$ with center at $P$, 
        \item[(c)] $\# \{\ell \subset \PP^3 \mid \ell \ \mathrm{is \ a \ line \ such \ that} \ P \in \ell \subset Y\} \le d(d - 1)$.
        \end{enumerate}
    \end{theorem}

   We conclude this subsection by recalling S{\o}rensen's bound, initially conjectured in \cite{SoT}, finally proved in \cite{BDH}, concerning the maximum size of the intersection of a surface and a non-degenerate Hermitian surface in $\PP^3$. 
 
 \begin{theorem} [S{\o}rensen's bound] \cite[Corollary 5.2]{BDH}\label{SoB}
Let $F \in \Fqt[x_0, x_1, x_2, x_3]$ be a homogeneous polynomial of degree $d$ and let $\U_3$ be a non-degenerate Hermitian surface in $\PP^3 (\Fqt)$. If $d \le q$, then
$$|V(F)(\Fqt) \cap \U_3| \le d(q^3 + q^2 - q) + q + 1.$$
Moreover, this bound is attained by $V(F)$ if and only if $V(F)$ is a union of $d$ planes $\Pi_1, \dots, \Pi_d$, each tangent to $\U_3$ and the planes intersect in a common line that is secant to $\U_3$. 
\end{theorem}

\subsection{Basics of linear codes}\label{linear code} Let $m$ and $k$ be positive integers. An $[m,k]_q$-code is a $k$-dimensional linear subspace of $\FF_q^m$, where $\FF_q^m$ is viewed as an $ m$-dimensional vector space over $\Fq$. Let $C$ be an $[m,k]_q$-code. The integers $m$ and $k$ are called the \textit{length} and the \textit{dimension} of $C$, respectively, and the elements of $C$ are referred to as \textit{codewords}. 

For a codeword $\mathbf{c}=(c_1,\dots,c_m)\in C$, the \textit{Hamming weight} of $\mathbf{c}$, denoted by $\wt(\mathbf{c})$, is defined as the number of indices $i\in\{1,\dots,m\}$ for which $c_i\neq 0$. The \textit{minimum distance} of $C$, denoted by $\mathsf{d}(C)$, is defined as
\[
\mathsf{d}(C)=\min\{\wt(\mathbf{c}): \mathbf{c}\in C,\; \mathbf{c}\neq \mathbf{0}\}.
\]
Codewords $\mathbf{c}\in C$ satisfying $\wt(\mathbf{c})=\mathsf{d}(C)$ are called \textit{minimum weight codewords} of $C$.

%%%%%%%%%%%%%%%%%%%%%%%%%%%%%%%%%%%%%%%%%%%%%%%%%%%%%%%%%%%%%%%%%%%%%%%%%%%%%%%%%%%%%%%%%%%%%%%%%%%%%%%%%%%%%%%%%%%%%%%%%%%%%%%%%%%%%%%%%%%%%%%%%%%   
\section{Intersection of a degenerate Hermitian variety of rank $n$ and a hypersurface}\label{degenerate of rank n}
Throughout this Section $V(F)$ will denote an $\Fqt$-hypersurface of degree $d$ with $d\leq q$ and by $P\U_{n-1}$ we denote the Hermitian variety of rank $n$ in $\PP^n$, defined over $\Fqt$. As mentioned above, the variety $P\U_{n-1}$ is a cone with vertex at an $\Fqt$-rational point $P$ and base as a non-degenerate Hermitian variety $\U_{n-1}$ in $\PP^{n-1}$.
%%%%%%%%%%%%%%%%%%%%%%%%%%%%%%%%%%%%%
\begin{lemma}\label{outside hyperplane}
 Let $n\geq 3$. If $\Sigma$ is an $\Fqt$-hyperplane in $\PP^n$  such that $P \not\in \Sigma$ and $\Sigma \subseteq  V(F)$, then $|P\U_{n-1} \cap V(F)\cap \Sigma^C(\Fqt)| \le (d-1)(q+1)q^{2n-4}$, where $\Sigma^C = \PP^n \setminus \Sigma$. 
\end{lemma}

\begin{proof}
    Since $\Sigma \subseteq V(F)$, it follows that $V(F) \cap \Sigma^C$ is an affine hypersurface of degree at most $d-1$. Furthermore, since $n\geq 3$, the degenerate Hermitian variety $P\U_{n-1}$ is an irreducible hypersurface of degree $q+1$. Consequently, the intersection $P\U_{n-1} \cap \Sigma^C$ is an irreducible affine hypersurface of degree $q+1$ contained in $\Sigma^C\cong \AA^n$. As $d\leq q$, the affine hypersurfaces $\left(V\left(F\right) \cap \Sigma^C\right)$ and $\left(P\U_{n-1} \cap \Sigma^C\right)$ have no common components, and hence, their intersection is a complete intersection of degree at most $(d-1)(q+1)$ and $\dim \left(P\U_{n-1} \cap V\left(F\right) \cap \Sigma^C\right) = n-2$. Thus Proposition \ref{lac} entails that $|P\U_{n-1} \cap V(F) \cap \Sigma^C(\Fqt)| \le (d-1)(q+1)q^{2n-4}$.
\end{proof}
%%%%%%%%%%%%%%%%%%%%%%%%%%%%%%%%%%%%%%%%%%

\begin{lemma}\label{missing P}
  Let $n\geq 3$. If $P\notin V(F)$, then $$|P\U_{n-1}\cap V(F)(\Fqt)|\leq d|\U_{n-1}(\Fqt)|< |\U_{n-1}(\Fqt)|+(d-1)(q+1)q^{2n-4}.$$   
\end{lemma}

\begin{proof}
    Since $P\notin V(F)$, any line of the cone $P\U_{n-1}$ that passes through $P$ intersects the hypersurface $V(F)$ in at most $d$ many $\Fqt$-rational points. Thus we conclude that $$|P\U_{n-1}\cap V(F)(\Fqt)|\leq d  |\U_{n-1}(\Fqt)|.$$ Now, a straightforward computation yields that 
    \begin{align}
        &\left(|\U_{n-1}(\Fqt)|+(d-1)(q+1)q^{2n-4} \right)- d|\U_{n-1}(\Fqt)|\notag\\
        &=\frac{d-1}{q^2-1}\left\{ q^{2n-4}(q^2-q-1)-(-1)^nq^{n-1}(q-1)+1 \right\}.\label{easy com}
    \end{align}
We consider two cases.

\textbf{Case 1:} \textit{$n$ is odd.} From the expression in the braces of Equation \eqref{easy com} we see that $ q^{2n-4}(q^2-q-1)+q^{n-1}(q-1)+1>0.$

\textbf{Case 2:} \textit{$n$ is even.} In this case the expression in the braces of Equation \eqref{easy com} becomes $q^{2n-4}(q^2-q-1)-q^{n-1}(q-1)+1$. Since $(q^2-q-1)-(q-1)=q(q-2)>0$ for $q\geq 3$ so it follows that $$q^{2n-4}(q^2-q-1)-q^{n-1}(q-1)+1>0$$ for $q\geq 3$. It is easy to check that the above expression remains positive for $q=2$, since $n\geq 3$. Combining the two cases yields the lemma.
\end{proof}

%%%%%%%%%%%%%%%%%%%%%%%%%%%%%%%%%%%%%%%%%%%%%
We are now ready to present one of the main theorems concerning an upper bound of the number of $\Fqt$-rational points in the intersection of a hypersurface and $P\U_{n-1}$.

\begin{theorem}\label{rank n}
  Let $n\geq 3$. We have $$|P\U_{n-1}\cap V(F)(\Fqt)|\leq \max \left\{|\U_{n-1}(\Fqt)|+(d-1)(q+1)q^{2n-4}, \ 1+q^2 M_{n-1}(d) \right\},$$ where $M_{n-1}(d):= \max_{V(G)}|\U_{n-1}\cap V(G)(\Fqt)|$ and the maximum is taken over all the $\Fqt$-hypersurfaces $V(G)$ of degree $d\leq q$ in $\PP^{n-1}$.   
\end{theorem}

\begin{proof}
The proof is divided into two cases, according as $P\in V(F)$ or $P\notin V(F)$.

    \textbf{Case 1: $P\notin V(F)$.} It follows from Lemma \ref{missing P} that $$|P\U_{n-1}\cap V(F)(\Fqt)|< |\U_{n-1}(\Fqt)|+(d-1)(q+1)q^{2n-4}.$$

    \textbf{Case 2: $P\in V(F)$.} We distinguish two subcases according to whether the hypersurface $V(F)$ contains an $\Fqt$-hyperplane that does not pass through the point $P$.

    \textit{Subcase 2a.} \textit{$V(F)$ contains an $\Fqt$-hyperplane not passing through $P$.} Suppose $V(F)$ contains an $\Fqt$-hyperplane $\Sigma$ such that $P\notin \Sigma$. It then follows from Corollary \ref{hyp cup} that the hyperplane $\Sigma$ intersects the singular Hermitian variety $P\U_{n-1}$ at a non-degenerate Hermitian variety $\U_{n-1}$. Since $\Sigma\subset V(F)$, we have $P\U_{n-1}\cap V(F)\cap \Sigma =\U_{n-1}$. Consequently, 
    \begin{align*}
        |P\U_{n-1}\cap V(F)(\Fqt)| &= |P\U_{n-1}\cap V(F)\cap \Sigma(\Fqt)|+|P\U_{n-1}\cap V(F)\cap \Sigma^C(\Fqt)|\\
        & \leq |\U_{n-1}(\Fqt)|+(d-1)(q+1)q^{2n-4},
        \end{align*}
        where the last inequality follows from Lemma \ref{outside hyperplane}.

        \textit{Subcase 2b.} \textit{$V(F)$ does not contain any $\Fqt$-hyperplane not passing through $P$.} We now define the sets 
        \begin{align*}
            \Gamma_1 &:=P\U_{n-1}\cap V(F)(\Fqt)\setminus \left\{P\right\} \ \text{and}\\
            \Gamma_2 &:= \left\{\Sigma: \Sigma \ \text{is an $\Fqt$-hyperplane},\ P\notin \Sigma \right\}.
        \end{align*}
Since there are $|\PP^{n-1}(\Fqt)|$ many $\Fqt$-hyperplanes through the point $P$ in $\PP^n$, it follows that $|\Gamma_2|=q^{2n}$. Consider the incidence set $$\I:=\left\{\left(Q,\Sigma\right)\in \Gamma_1\times \Gamma_2 \ | \ Q\in \Sigma  \right\}.$$ For a chosen point $Q\in \Gamma_1$ we see that 
\begin{align}
\#\left\{ \Sigma\in \Gamma_2: Q\in \Sigma \right\} &= \#\left\{ \Sigma\ |\  Q\in \Sigma \right\}-\#\left\{ \Sigma\ |\ \text{the line}\  PQ\subset \Sigma   \right\}\notag\\
         &=|\PP^{n-1}\left(\Fqt\right)|- |\PP^{n-2}\left(\Fqt\right)|\notag \\
         &=q^{2n-2}\label{count hyp}.
\end{align}
We will count $|\I|$ in two ways. First, we observe that, 
\begin{equation}\label{first count}
  |\I|=\sum_{Q\in \Gamma_1}\#\left\{ \Sigma\in\Gamma_2: Q\in \Sigma \right\}=|\Gamma_1|q^{2n-2},  
\end{equation}
where the last equality follows from \eqref{count hyp}. On the other hand, since $P\notin \Sigma$ for any hyperplane in $\Gamma_2$, it follows that $\Sigma\cap P\U_{n-1}$ is a non-degenerate Hermitian variety $\U_{n-1}$ in $\PP^{n-1}$. Moreover, since by assumption $V(F)$ does not contain any hyperplane $\Sigma$ of $\Gamma_2$, the intersection $V(F)\cap \Sigma$ is an $\Fqt$-hypersurface of degree $d$ in $\Sigma\cong\PP^{n-1}$. Consequently, we count that 
\begin{align}
    |\I|&=\sum_{\Sigma\in \Gamma_2}\#\{ Q\in \Gamma_1: Q\in \Sigma \}\notag\\
        &= \sum_{\Sigma\in \Gamma_2} |\Sigma\cap P\U_{n-1}\cap V(F)\left(\Fqt\right)| \notag\\
        &=\sum_{\Sigma\in \Gamma_2}|\U_{n-1}\left(\Fqt\right)\cap \left(\Sigma\cap V(F)  \right)|\notag\\
        &\leq |\Gamma_2| M_{n-1}(d)=q^{2n}M_{n-1}(d).\label{second count}
\end{align}
Thus it follows from comparing \eqref{first count} and \eqref{second count} that $|\Gamma_1|\leq q^2 M_{n-1}(d)$, which implies that $$|P\U_{n-1}\cap V(F)(\Fqt)|\leq 1+q^2M_{n-1}(d).$$ Now the theorem follows from the two cases.
\end{proof}
%%%%%%%%%%%%%%%%%%%%%%%%%%%%%%%%%%%%%%%%%%%%%%
\subsection{The case $n=2$}\label{n=2}
\begin{theorem}\label{2-main-theorem}
 If $V(F)$ is an $\Fqt$-curve of degree $d\leq q$ in $\PP^2$, then $$|P\U_1\cap V(F)(\Fqt)|\leq dq^2+1.$$ Moreover, the bound is attained if and only if $V(F)$ is union of $d$ lines, each contained in the curve $P\U_1$. 
\end{theorem}
\begin{proof}
    Since $P\U_1(\Fqt)$ is the union of $q+1$ many $\Fqt$ lines, each passing through the point $P$, the Theorem follows directly from Serre's inequality in Theorem \ref{serre}.
\end{proof}

\subsection{The case $n=3$}\label{n=3}In this subsection we restrict ourself to the case when $n=3$ and characterize the $\Fqt$-surfaces that shares the maximum number of $\Fqt$-rational points with the rank $3$ degenerate Hermitian surface $P\U_2$. To begin, we establish an optimal upper bound for this maximum. Although the following theorem first appeared in \cite[Theorem 3.3]{DM}, we here classify the surfaces attaining the maximum number of $\Fqt$-rational points common to $P\U_2$. The bound stated in the following theorem is attained; see \cite[Remark 3.4]{DM}.

%%%%%%%%%%%%%%%%%%%%%%%%%%%%%%%%%%%%%%%%%%%%%%
\begin{theorem}
\label{surface and rank n}
If $V(F)$ is an $\Fqt$-surface of degree $d\leq q$ in $\PP^3$, then $$|P\U_2\cap V(F)(\Fqt)|\leq 1+q^2d(q+1).$$   
\end{theorem}
\begin{proof}
 Since $\U_2$ is an absolutely irreducible curve and $d\leq q$, it follows from B\'{e}zout's theorem that $$M_2(d)=\max_{V(G)}|\U_2\cap V(G)(\Fqt)|=d(q+1),$$ where the maximum is taken over all the plane curves $V(G)$ defined over $\Fqt$ of degree $d\leq q$. It is worth noting that there exist $\Fqt$-curves of degree $d$ that share exactly $d(q+1)$ many $\Fqt$-rational points with the non-degenerate Hermitian curve $\U_2$. For instance, if $\C$ is a curve consisting of $d$  distinct concurrent secant $\Fqt$-lines passing through a point of $\PP^2(\Fqt) \setminus \U_2$, then each of these lines meets the Hermitian curve $\U_2$ in exactly $q+1$ many $\Fqt$-rational points, as mentioned in Lemma \ref{line}. Consequently, $\C$ intersects $\U_2$ in precisely $d(q+1)$ many such points. Although producing such examples is easy, it appears that the problem of characterizing all the $\Fqt$-curves that attain exactly $d(q+1)$ many $\Fqt$-rational points common to $\U_2$ is a challenging one and still unresolved. To this end, we refer to \cite{BDMN} for such examples of curves. 
 
 Now since $|\U_2(\Fqt)|+(d-1)(q+1)q^2=dq^3+(d-1)q^2+1<1+q^2M_2(d)$, we conclude from Theorem \ref{rank n} that $$|P\U_2\cap V(F)(\Fqt)|\leq 1+q^2d(q+1).$$ 
\end{proof}
%%%%%%%%%%%%%%%%%%%%%%%%%%%%%%%%%%%%%%%%%%%%%%%%%%
\begin{lemma}\label{line confi} If $|P\U_2\cap V(F)(\Fqt)|=1+q^2d(q+1)$ then any $\Fqt$-line $\ell\subset P\U_2$ satisfies precisely one of the following:
   \begin{enumerate}
       \item[(a)] $\ell\subset V(F)$,
       \item [(b)] $\ell\cap V(F)(\Fqt)=\{P\}$.
   \end{enumerate}
\end{lemma}
\begin{proof}
Since $|P\U_2\cap V(F)(\Fqt)|=1+q^2d(q+1)$ it follows from the proof of Theorem \ref{rank n} that $P\in V(F)$ and $V(F)$ does not contain any $\Fqt$-plane not passing through $P$. Suppose that $\ell$ is an $\Fqt$-line contained in $P\U_2$ such that $\ell\not\subset V(F)$, and let $Q\in \ell\cap V(F)(\Fqt)$ with $Q\neq P$. Let $\Pi_0$ be an $\Fqt$-plane such that $Q\in \Pi_0$ but $P\notin \Pi_0$. Consequently, corollary \ref{hyp cup} entails that the plane $\Pi_0$ intersects the cone $P\U_2$ at a non-singular Hermitian curve $\U_2$. Let $m$ be the tangent line to $\U_2\subset \Pi_0$ at the point $Q$. Evidently, $m\cap P\U_2\cap V(F)(\Fqt)=\{Q\}$. We now define the book of $\Fqt$-planes around the line $m$ as $$\B(m):=\left\{\Pi: \Pi \ \text{is an}\ \Fqt\text{-plane and} \ m\subset \Pi  \right\}.$$ Note that $|\B(m)|=q^2+1$ and $\Pi_0\in \B(m)$. Let $\Pi_1:= \langle m,\ell\rangle$ be the plane defined by the span of two lines $m$ and $\ell$. 

\textbf{Claim:} $\Pi_1\cap P\U_2(\Fqt)=\ell(\Fqt)$.

\textit{Proof of claim:} First we observe that $\ell(\Fqt)\subseteq \Pi_1\cap P\U_2(\Fqt)$. Let $R\in \Pi_1\cap P\U_2(\Fqt)\setminus \ell $. Therefore the line $\ell':=\langle P,R\rangle$ is contained in $\Pi_1\cap P\U_2$. If $\ell'$ meets $\Pi_0$ at $Q'$, the line $m$ must passes through the point $Q'$, which implies that $m$ is not a tangent to $\U_2$. This proves our claim.

Let $\Pi_2,\dots,\Pi_{q^2}$ denote the remaining planes of $\mathcal{B}(m)$ besides $\Pi_0$ and $\Pi_1$.
Since $\Pi_1$ is the unique plane of $\B(m)$ passing through $P$, we have $\Pi_i\cap P\U_2=\U_2$ for all $i=0,2,\dots,q^2$. Additionally, because $V(F)$ contains no plane that does not pass through $P$, the intersection $V(F)\cap \Pi_i$ is a plane curve of degree $d$, for all $i=0,2,\dots,q^2$. Thus B\'{e}zout's theorem entails that $$|\Pi_i\cap P\U_2\cap V(F)(\Fqt)|\leq d(q+1) \ \text{for}\ i=0,2,\dots, q^2.$$ Furthermore, since $\ell \not\subset V(F)$ we have $|\Pi_1\cap P\U_2\cap V(F)(\Fqt)|=|\ell\cap V(F)(\Fqt)|\leq d$. Hence, we conclude that 
\begin{align*}
    |P\U_2\cap V(F)(\Fqt)|&= |m\cap P\U_2\cap V(F)(\Fqt)|+\sum_{i=0}^{q^2}|\Pi_i\cap P\U_2\cap V(F)\setminus m (\Fqt)|\\
    &=|\{Q\}|+|\Pi_1\cap P\U_2\cap V(F)\setminus m(\Fqt)|+\sum_{\substack{i=0 \\ i\neq 1}}^{q^2}|\Pi_i\cap P\U_2\cap V(F)\setminus m (\Fqt)| \\
    &\leq 1+(d-1)+q^2\left(d\left(q+1\right)-1\right)\\
    &=q^2d(q+1)-(q^2-d)\\
    &<1+q^2d(q+1),
\end{align*}
a contradiction to our hypothesis.
\end{proof}

%%%%%%%%%%%%%%%%%%%%%%%%%%%%%%%%%%%%%%%%%%%%%%%    
\begin{theorem}\label{attain con line}
  $|P\U_2(\Fqt)\cap V(F)(\Fqt)|=1+q^2d(q+1)$ if and only if $P\U_2\cap V(F)(\Fqt)$ is union of $d(q+1)$ many $\Fqt$-lines of $P\U_2$.    
\end{theorem}
\begin{proof}
  If $P\U_2\cap V(F)(\Fqt)$ is union of $d(q+1)$ lines of $P\U_2$, then $|P\U_2(\Fqt)\cap V(F)(\Fqt)|=1+q^2d(q+1)$. The other implication is an immediate consequence of the Lemma \ref{line confi}.  
\end{proof}
%%%%%%%%%%%%%%%%%%%%%%%%%%%%%%%%%%%%%%%%%%%%%%%
Now we classify the $\Fqt$-surfaces based on the bound presented in Theorem \ref{surface and rank n}. To this end, we first prove the following lemma, which will be imperative in proving Theorem \ref{cofi of V(F)}.
\begin{lemma}\label{plane cut}
 Let $\Pi$ be an $\Fqt$-plane passing through $P$. Then $\Pi\cap P\U_2(\Fqt)$ contains at most $q+1$ lines of the cone $P\U_2$.   
\end{lemma}
\begin{proof}
 Let $\Pi_0$ be an $\Fqt$-plane such that $P\notin \Pi_0$ and let $\ell$ be the line of intersection of these two planes $\Pi$ and $\Pi_0$. Since $P\notin \Pi_0$, it follows that the intersection $P\U_2\cap \Pi_0$ is a nonsingular plane Hermitian curve $\U_2$ contained in the plane $\Pi_0$. Furthermore, since $\U_2$ contains no line, the line $\ell$ is either a tangent or a secant line of $\U_2$. Now if $\ell$ is tangent to $\U_2$ at an $\Fqt$-point $Q$, then $\Pi\cap P\U_2(\Fqt)$ is the line $PQ(\Fqt)$. On the other hand, if $\ell$ is secant to $\U_2$, then it intersects $\U_2$ at $q+1$ distinct $\Fqt$-rational points, say $\{Q_0,\dots, Q_q\}$. In this case $\Pi\cap P\U_2(\Fqt)$ is the union of $q+1$ many $\Fqt$-lines $PQ_i$, where $i=0,\dots,q$.  
\end{proof}
%%%%%%%%%%%%%%%%%%%%%%%%%%%%%%%%%%%%%%%%%%%%%%%%%%

\begin{theorem}\label{cofi of V(F)}
   Let $V(F)$ be an $\Fqt$-surface of degree $d\leq q$ in $\PP^3$. Then $|P\U_2\cap V(F)(\Fqt)|=1+q^2d(q+1)$ if and only if $V(F)$ is a cone over an $\Fqt$-curve $\C$ with vertex at the point $P$ such that $|\C\cap P\U_2(\Fqt)|=d(q+1)$.  
\end{theorem}
\begin{proof}
   It follows from the Theorem \ref{attain con line} that, $P\U_2\cap V(F)(\Fqt)$ is union of $d(q+1)$ many $\Fqt$-lines. Furthermore, the \textit{Case 2} of Theorem \ref{rank n} implies that if the surface $V(F)$ contains an $\Fqt$-plane, then that plane must pass through the point $P$, otherwise $$|P\U_2\cap V(F)(\Fqt)|\leq dq^3+(d-1)q^2+1< 1+dq^2(q+1).$$ Now since $d(d-1)<d(q+1)$, we conclude from Theorem \ref{Datta} that either the surface $V(F)$ contains a cone over a plane curve defined over $\Fqt$ with center at $P$, or $V(F)$ contains a plane defined over $\Fqt$. Therefore, we can express $V(F)$ as \begin{equation}\label{expression}
   V(F)=\left(\bigcup_{i=1}^r \Pi_i\right)\cup \left(\bigcup_{j=1}^s P\widehat{\C_j}\right)\cup V(G),
   \end{equation}where each $\Pi_i$ is a plane defined over $\Fqt$ passing through $P$, each $P\widehat{\C_j}$ denotes a cone over a plane curve $\widehat{\C_j}$ with vertex at $P$ and $V(G)$ is an $\Fqt$-surface  that contains neither a cone over a plane curve with vertex at $P$ nor a plane defined over $\Fqt$. Let $\deg(\widehat{\C_j})=d_j$ and $\deg(G)=d_G$. Then from the expression \eqref{expression} of $V(F)$ we obtain \begin{equation}\label{deg inq} r+d_G+\sum_{j=1}^s d_j\leq d.    
   \end{equation}
   
   \textbf{Claim:} $d_G=0$.

  \textit{Proof of claim:} Suppose that $d_G\neq 0$. Since $V(G)$ contains no $\Fqt$-plane, it follows that $d_G>1$. Depending on $P\in V(G)$ or $P\notin V(G)$, we distinguish two cases.

  \textbf{Case 1:} $P\notin V(G)$. From lemma \ref{plane cut} we conclude that $$\#\left\{\ell: \ell \ \text{is an}\ \Fqt\text{-line}, \ell\subset P\U_2\cap \left(\bigcup_{i=1}^r\Pi_i\right)\right\}\leq r(q+1).$$ Let $\widehat{\Pi}_j$ be the plane of the curve $\widehat{\C_j}$. Since $P\notin \widehat{\Pi}_j$, therefore $P\U_2\cap\widehat{\Pi}_j$ is a non-degenerate Hermitian curve $\U_2$. Moreover, since $d_j\leq q$ and $\U_2$ is absolutely irreducible, there are no common components of $\widehat{\C_j}$ and $\U_2$. Consequently, by B\'{e}zout's theorem we have $|\widehat{\C_j}\cap \U_2(\Fqt)|\leq d_j(q+1)$ and this implies that there are at most $d_j(q+1)$ many $\Fqt$-lines of $P\widehat{\C_j}$ common to $P\U_2$. Thus $$\#\left\{\ell: \ell \ \text{is an}\ \Fqt\text{-line}, \ell\subset P\U_2\cap \left(\bigcup_{j=1}^s P\widehat{\C_j}\right)\right\}\leq (q+1)\sum_{j=1}^s d_j.$$ Now since $P\notin V(G)$, so the surface $V(G)$ contains no line of $P\U_2$. Consequently, 
  \begin{align*}
     \#\left\{\ell: \ell \ \text{is an}\ \Fqt\text{-line}, \ell\subset P\U_2\cap V(F)\right\} &\leq r(q+1)+(q+1)\sum_{j=1}^s d_j\\
     &< d(q+1),
  \end{align*}
 where the last inequality follows from the inequality \eqref{deg inq} together with the fact that $d_G>1$.

\textbf{ Case 2:} $P\in V(G)$. It follows from Theorem \ref{Datta} that the surface $V(G)$ contains at most $d_G(d_G-1)$ lines passing through the point $P$. In the same vein as Case 1, we conclude that
\begin{align*}
  \#\left\{\ell: \ell \ \text{is an}\ \Fqt\text{-line}, \ell\subset P\U_2\cap V(F)\right\} &\leq  r(q+1)+(q+1)\left(\sum_{j=1}^s d_j\right)+d_G(d_G-1)\\
  &= (q+1)\left(r+d_G+\sum_{j=1}^sd_j\right)-d_G(q-d_G+2)\\
  &\leq d(q+1)-d_G(q-d_G+2) \ \ \left(\text{using}\ \eqref{deg inq}\right)\\
  &<d(q+1),
\end{align*}
where the last inequality follows since $0<d_G\leq q$. 

Thus, in both cases, we arrive at a contradiction concerning the number of lines contained in $V(F)\cap P\U_2$. As a result $d_G=0$, which in turn implies that $$V(F)=\left(\bigcup_{i=1}^r \Pi_i\right)\cup \left(\bigcup_{j=1}^s P\widehat{\C_j}\right).$$ Now since each plane $\Pi_i$ through $P$ can be realized as a cone over a line $\ell_i\subset\Pi_i$ with vertex at $P$, where $P\notin \ell_i$, it follows that $V(F)$ is a cone over a curve $\C:=\Big(
\bigcup_{i=1}^r \ell_i
\Big)\cup\Big(\bigcup_{j=1}^s \widehat{\mathcal{C}_j}\Big)$ with vertex at $P$. This completes our proof.

\end{proof}

%%%%%%%%%%%%%%%%%%%%%%%%%%%%%%%%%%%%%%%%%

\subsection{The case $n=4$}\label{n=4} In this subsection, we study the intersection of a degenerate Hermitian threefold of rank $4$ and a threefold of degree $d\leq q$ and characterize the threefolds of degree at most $q-1$ based on their intersection number with the Hermitian threefold. Throughout, $V(F)$ will denote a threefold of degree $d\leq q$ and $P\U_3$ is the Hermitian threefold of rank $4$.
%%%%%%%%%%%%%%%%%%%%%%%%%%%%%%%%%%%%%%%%%%%%%%%
\begin{theorem}\label{degenerate and threefold} 
  We have $$|P\U_3\cap V(F)(\Fqt)|\leq 1+q^2\left(d(q^3+q^2-q)+q+1\right).$$   
\end{theorem}
\begin{proof}
  It follows from S{\o}rensen's bound in Theorem \ref{SoB} that $$M_3(d)=\max_{V(G)}|\U_3\cap V(G)(\Fqt)|=d(q^3+q^2-q)+q+1,$$ where the maximum is taken over all the $\Fqt$-surfaces $V(G)$ of degree $d\leq q$. Now, a straightforward computation yields 
  \begin{equation}\label{inq: threefold}
      \left(1+q^2M_3(d)\right)-\left(|\U_3(\Fqt)|+(d-1)(q+1)q^4 \right)= q^3(q-d)\geq 0.
  \end{equation}
  Thus, we conclude from Theorem \ref{rank n} that $$|P\U_3\cap V(F)(\Fqt)|\leq 1+q^2M_3(d).$$
\end{proof}
%%%%%%%%%%%%%%%%%%%%%%%%%%%%%%%%%%%%%%%%%%%%%%
\begin{remark}\label{when attained}
   \normalfont It is not difficult to construct an $\Fqt$-threefold of degree $d$ that intersects $P\U_3$ at exactly $1+q^2M_3(d)$ many $\Fqt$-rational points. Let $\Sigma$ be an $\Fqt$-hyperplane in $\PP^4$ such that $P\notin \Sigma$. Then it follows from corollary \ref{hyp cup} that the hyperplane $\Sigma$ intersects $P\U_3$ at a non-degenerate Hermitian surface $\U_3$. Take a surface $\s$ of degree $d$ in $\Sigma\cong \PP^3$ such that $|\s\cap \U_3(\Fqt)|=M_3(d)$. In fact, such a surface $\s$ is fully characterized by Theorem \ref{SoB}. It is now evident that the cone $P\s$, which is a threefold of degree $d$ defined over $\Fqt$, intersects $P\U_3$ at exactly $1+q^2M_3(d)$ many $\Fqt$-rational points. 
   
   Indeed, using the same construction above, one can construct an $\Fqt$-hypersurface $V(F)$ of degree $d\leq q$ in $\PP^n(\Fqt)$ with $n\geq 5$ such that $|P\U_{n-1}\cap V(F)(\Fqt)|=1+q^2M_{n-1}(d)$.
\end{remark}
%%%%%%%%%%%%%%%%%%%%%%%%%%%%%%%%%%%%%%%%%%%%%%%
We now present a characterization of threefolds of degree at most $q-1$ passing through the point $P$.
\begin{theorem}\label{weak condition}
  Let $V(F)$ be an $\Fqt$-threefold of degree $d\leq q-1$ with $P\in V(F)$. Then $|P\U_3\cap V(F)(\Fqt)|= 1+q^2M_3(d)$ if and only if for every $\Fqt$-hyperplane $\Sigma$, not passing through $P$, the surface $V(F)\cap \Sigma$ is union of $d$ tangent planes to $\U_3\subset \Sigma$ such that the tangent planes meet in a secant line of $\U_3$.
\end{theorem}
\begin{proof}
 It follows from the Equation \eqref{inq: threefold} that if $d<q$ then $$|\U_3(\Fqt)|+(d-1)(q+1)q^4< 1+q^2M_3(d).$$
 
 \textbf{Claim:} \textit{$|P\U_3\cap V(F)(\Fqt)|= 1+q^2M_3(d)$ if and only if $|\Sigma\cap P\U_3\cap V(F)(\Fqt)|=M_3(d)$ for every $\Fqt$-hyperplane $\Sigma$, not passing through the point $P$.}

 \textit{Proof of claim:} Suppose that $|P\U_3\cap V(F)(\Fqt)|= 1+q^2M_3(d)$. First, we note from the proof of Theorem \ref{rank n} that this equality can occur only when $V(F)$ does not contain any hyperplane not passing through $P$ and $P\in V(F)$. If there exists an $\Fqt$-hyperplane $\Sigma'$ such that $|\Sigma'\cap P\U_3\cap V(F)(\Fqt)|<M_3(d)$ then it follows from inequality \eqref{second count} that $$|\I|<q^8M_3(d),$$ which leads to $|P\U_3\cap V(F)(\Fqt)|<1+q^2M_3(d)$, a contradiction.

 Conversely, suppose that $|\Sigma\cap P\U_3\cap V(F)(\Fqt)|=M_3(d)$ for every $\Fqt$-hyperplane $\Sigma$, not passing through the point $P$. Consequently, for each such hyperplane $\Sigma$ we have $\Sigma \not\subset V(F)$, otherwise $|\Sigma\cap P\U_3\cap V(F)(\Fqt)|=|\Sigma\cap P\U_3(\Fqt)|=|\U_3(\Fqt)|> M_3(d)$. Since $P\in V(F)$, from \textit{Subcase 2b} of the proof of theorem \ref{rank n} we conclude that $|P\U_3\cap V(F)(\Fqt)|=1+q^2M_3(d)$. This proves our claim. Now the Theorem follows from Theorem \ref{SoB}. 
    
\end{proof}

%%%%%%%%%%%%%%%%%%%%%%%%%%%%%%%%%%%%%%%%%%%%%%%%%
The above Theorem gives a characterization of threefolds passing through $P$. Next, we give another characterization for a threefold of degree at most $q-1$. To this end, we first prove several preliminary lemmas.

\begin{lemma}\label{hyperplane cut}
 Let $\Sigma$ be an $\Fqt$-hyperplane passing through the point $P$. Then $\Sigma\cap P\U_3(\Fqt)$ is either a cone over a non-degenerate Hermitian curve with vertex at $P$ or a cone over $q+1$ concurrent lines having vertex at $P$.   
\end{lemma}
\begin{proof}
Let $\Sigma_0$ be a hyperplane defined over $\Fqt$ such that $P\notin \Sigma_0$. It follows from Corollary \ref{hyp cup} that $\Sigma_0$ and $P\U_3$ intersect in a non-degenerate Hermitian surface $\U_3$ in $\Sigma_0$. Let $\Pi$ be the plane of intersection of the two hyperplanes $\Sigma$ and $\Sigma_0$.

\textbf{Claim:} \textit{$\Sigma\cap P\U_3(\Fqt)$ is a cone over the plane curve $\Pi\cap\U_3(\Fqt)$.}

\textit{Proof of claim:} Suppose that $Q$ is an $\Fqt$-rational point of $\Sigma\cap P\U_3$ other than $P$. Then the line $PQ$, joining the points $P$ and $Q$, is contained in $P\U_3\cap \Sigma$ and $PQ$ intersects $\U_3\subset \Sigma_0$ at an $\Fqt$-point, say $R$. Consequently, $R\in \Sigma\cap\Sigma_0\cap\U_3=\Pi\cap\U_3$ and $Q$ lies on the line $PR$. This proves the claim.

 Since the plane $\Pi$ is either non-tangent or tangent to the Hermitian surface $\U_3$, it follows from Theorem \ref{hyperplane section} that $\Pi \cap \U_3$ is either a nonsingular Hermitian curve or a union of $q+1$ concurrent lines. The Lemma now follows from the claim.  
\end{proof}

%%%%%%%%%%%%%%%%%%%%%%%%%%%%%%%%%%%%%%%%%%%%%%%%%%

\begin{lemma}\label{small lem}
 Let $\Sigma_0$ be an $\Fqt$-hyperplane not passing through $P$ and $\Sigma_1$ is an $\Fqt$ hyperplane with $P\in \Sigma_1$, such that the plane $\Sigma_0\cap \Sigma_1$ is tangent to $\U_3\subset \Sigma_0$ at an $\Fqt$-point $R$. If the line $PR$ is not contained in the threefold $V(F)$ and $P\in V(F)$, then $$|\Sigma_1\cap P\U_3\cap V(F)(\Fqt)|\leq dq^3+dq^2+d.$$   
\end{lemma}
\begin{proof}
 Since $\Sigma_0\cap \Sigma_1$ is tangent to $\U_3$ at the point $R$, therefore $\Sigma_0\cap \Sigma_1\cap \U_3$ is union of $q+1$ lines, say $m_0,\dots, m_q$, each passing through $R$. Consequently, it follows from Lemma \ref{hyperplane cut} that $\Sigma_1\cap P\U_3(\Fqt)$ is cone over these $q+1$ lines passing through the point $R$, i.e., $$\Sigma_1\cap P\U_3(\Fqt)=\bigcup_{i=0}^q \langle P, m_i\rangle (\Fqt),$$ where $\langle P, m_i\rangle$ denotes the plane passing through $P$ and containing the line $m_i$. Furthermore, since $PR\not\subset V(F)$, none of the planes $\langle P,m_i \rangle$ is contained in $V(F)$. As a result, from Serre's inequality (cf. Theorem \ref{serre}), we have $|\langle P,m_i \rangle \cap V(F)(\Fqt)|\leq dq^2+1$. 
 Hence, we conclude that 
 \begin{align*}
 |\Sigma_1\cap P\U_3\cap V(F)(\Fqt)|&= |PR\cap V(F)(\Fqt)|+\sum_{i=0}^q |\langle P,m_i\rangle \cap V(F)\setminus PR(\Fqt)|\\
 &\leq d+(q+1)(dq^2+1-1) \ \left(\text{since}\ P\in V(F)(\Fqt)\right)\\
 &=dq^3+dq^2+d.
 \end{align*}
\end{proof}
%%%%%%%%%%%%%%%%%%%%%%%%%%%%%%%%%%%%%%%%%%%%%
\begin{lemma}\label{plane involved}
    Let $V(F)$ be an $\Fqt$-threefold of degree $d\leq q-1$ and $|P\U_3\cap V(F)(\Fqt)|=1+q^2M_3(d)$. If $\ell$ is an $\Fqt$-line such that $\ell\subset V(F)\cap \U_3$, then the plane $\Pi:=\langle P,\ell\rangle$ is contained in $V(F)$.
\end{lemma}
\begin{proof}
If all the $\Fqt$-rational points of $\Pi$ are contained in the threefold $V(F)$, then clearly $\Pi\subset V(F)$. In fact, a plane curve of degree $d\leq q$ contains at most $dq^2+1$ many $\Fqt$-rational points, due to Serre's inequality. Let $Q\in \Pi(\Fqt)$ such that $Q\notin V(F)$ and let $m$ be the line joining the points $P$ and $Q$. Suppose $m$ intersects $\ell$ at the point $R$. 

Let $\Sigma_0$ be an $\Fqt$-hyperplane passing through $\ell$ with $P\notin \Sigma_0$. We consider the tangent plane $\T_R(\U_3)$ to $\U_3\subset \Sigma_0$ at the point $R$. Let $\B\left(\T_R(\U_3)\right)$ denote the book of $\Fqt$-hyperplanes around the plane $\T_R(\U_3)$, defined by $$\B\left(\T_R(\U_3)\right):=\left\{ \Sigma: \Sigma \ \text{is an} \ \Fqt\text{-hyperplane and}\ \T_R(\U_3)\subset \Sigma\right\}.$$ Note that, $|\B\left(\T_R(\U_3)\right)|=q^2+1$ and $\Sigma_0\in \B\left(\T_R(\U_3)\right)$. Let $\Sigma_1$ be the hyperplane of $\B\left(\T_R(\U_3)\right)$ passing through $P$. Evidently, $PR\not \subset V(F)$, since $Q\notin V(F)$ and $\Sigma_0\cap\Sigma_1=\T_R(\U_3)$. Thus, it follows from Lemma \ref{small lem} that $$|\Sigma_1\cap P\U_3\cap V(F)(\Fqt)|\leq dq^3+dq^2+d.$$ Furthermore, since $|P\U_3\cap V(F)(\Fqt)|=1+q^2M_3(d)$, from Theorem \ref{rank n} (\textit{Subcase 2b}) we conclude that the threefold $V(F)$ does not contain any hyperplane not passing through $P$. Let $\Sigma_2,\dots,\Sigma_{q^2}$ be the other hyperplanes of $\B(\T_R(\U_3))$ distinct from $\Sigma_0,\Sigma_1$.  Consequently, S{\o}rensen's bound from Theorem \ref{SoB} implies that $$|\Sigma_i\cap P\U_3\cap V(F)(\Fqt)|\leq M_3(d),\ \text{for}\ i=0,2,\dots,q^2.$$  Additionally, since $\ell\subset \T_R(\U_3)\cap V(F)$, we have $$|\T_R(\U_3)\cap P\U_3\cap V(F)(\Fqt)|\geq q^2+1.$$ Thus, it follows that 
\begin{align*}
    |P\U_3\cap V(F)(\Fqt)|&=|\Sigma_1\cap P\U_3\cap V(F)(\Fqt)|+\sum_{\substack{i=0 \\ i\neq 1}}^{q^2}|\Sigma_i\cap P\U_3\cap V(F)\setminus \T_R(\U_3)(\Fqt)|\\
    &\leq dq^3+dq^2+d+q^2(M_3(d)-q^2-1)\\
    &= dq^5+(d-1)q^4+q^3+dq^2+d\\
    &< 1+q^2M_3(d)=dq^5+dq^4-(d-1)q^3+q^2+1,    
\end{align*} where the last inequality follows since $d\leq q-1$. This completes the proof of the Lemma.
\end{proof}

%%%%%%%%%%%%%%%%%%%%%%%%%%%%%%%%%%%%%%%%%%%%%%%%%%
We are now ready to prove the main Theorem of this subsection. This will classify all the $\Fqt$-threefolds of degree at most $q-1$ that share the maximum number of $\Fqt$-rational points with $P\U_3$. 

\begin{theorem}\label{main n=4}
 Let $V(F)$ be an $\Fqt$-threefold of degree $d\leq q-1$ in $\PP^4$. Then $|P\U_3(\Fqt)\cap V(F)(\Fqt)|=1+q^2M_3(d)$ if and only if $P\U_3\cap V(F)(\Fqt)$ is union of $M_3(d)$ many $\Fqt$-lines of the cone $P\U_3$, each passing through $P$.    
\end{theorem}
\begin{proof}
 If $P\U_3\cap V(F)(\Fqt)$ is union of $M_3(d)$ many concurrent $\Fqt$-lines then certainly $|P\U_3(\Fqt)\cap V(F)(\Fqt)|=1+q^2M_3(d)$. 
 
 Conversely, suppose that $|P\U_3(\Fqt)\cap V(F)(\Fqt)|=1+q^2M_3(d)$. If $P\notin V(F)$, then Case 1 of theorem \ref{rank n} entails that, $|P\U_3\cap V(F)(\Fqt)|<|\U_3(\Fqt)|+(d-1)(q+1)q^4<1+q^2M_3(d)$. Consequently, $P\in V(F)$. It now follows from Theorem \ref{weak condition} that if $\Sigma$ is an $\Fqt$-hyperplane that does not pass through $P$, then $$V(F)\cap \Sigma=\bigcup_{i=1}^d \Pi_i,$$ where each plane $\Pi_i$ is tangent to $\U_3\subset \Sigma$ at some point, say $R_i$ and the line $m:=\cap_{i=1}^d\Pi_i$ is secant to $\U_3$. Since $\Pi_i$ is tangent to $\U_3$ at $R_i$, the intersection $\Pi_i\cap \U_3$ is union of $q+1$ lines, say $\ell_{i0},\dots,\ell_{iq}$, each passing through $R_i$. Thus, from Lemma \ref{plane involved} we conclude that the plane $\langle P,\ell_{ij} \rangle$ is contained in $V(F)$, for every $j=0,\dots,q$. Thus, $V(F)$ contains all the lines of $P\U_3$ in the plane $\langle P,\ell_{ij} \rangle$, for all $i=1,\dots,d$ and $j=0,\dots,q$. Since $|V(F)\cap\Sigma\cap\U_3(\Fqt)|=M_3(d)$ the Theorem follows. 
\end{proof}
%%%%%%%%%%%%%%%%%%%%%%%%%%%%%%%%%%%%%%%%%%%%%%%%%
\subsection{The case $n\geq 5$}\label{bigger than 5}

It is evident from Theorem \ref{rank n} that determining an upper bound on the maximum number of $\Fqt$-rational points in the intersection of a hypersurface and a rank $n$ degenerate Hermitian variety in $\PP^n$ reduces to evaluating the quantity $M_{n-1}(d)$. Consequently, for the study of the case $n\geq 5$, one needs to evaluate $M_i(d)$ for all $i\geq 4$. However, as the dimension $n$ of the underlying projective space grows, determining $M_{n-1}(d)$ becomes increasingly difficult. To this end, in 2011, Edoukou, Ling and Xing proposed a conjecture in \cite{ELX} that can be reformulated as follows: 
\begin{conjecture}\cite[Conjecture 2 (ii)]{ELX}\label{ELXc}
   Let $n\geq 3$ and $\U_n$ be a non-degenerate Hermitian variety in $\PP^n(\Fqt)$. If $V(F)$ is a hypersurface of degree $d\leq q$, then 
   $$M_n(d)=\max_{V(F)} |V(F)(\Fqt)\cap \U_n|=\begin{cases}
       d|\U_{n-1}(\Fqt)|-(d-1)|\U_{n-2}(\Fqt)| \ \text{if} \ n \ \text{is even}\\
       (dq^2-d+1)|\U_{n-2}(\Fqt)|+d\ \text{if} \ n \ \text{is odd},
   \end{cases}$$ 
   and this maximum is attained if and only if the hypersurface $V(F)$ is union of $d$ distinct hyperplanes $\Sigma_1, \Sigma_2,\dots, \Sigma_d$ such that each of the hyperplanes contains a common $(n-2)$-dimensional linear space $\Pi_{n-2}$  and $\Pi_{n-2}\cap \U_n$ is a non-degenerate Hermitian variety and: 

     --- If $n$ is even, the $d$ hyperplanes $\Sigma_1,\dots, \Sigma_d$ are non-tangent to $\U_n$.

     --- If $n$ is odd, the $d$ hyperplanes $\Sigma_1,\dots, \Sigma_d$ are tangent to $\U_n$. 
\end{conjecture}
Conjecture \ref{ELXc} is known to hold for $d=1$, as established by Bose and Chakravarti in their seminal works  \cite{BC, C}. When $n=3$, the conjecture has been completely settled through a series of papers: The case $d=2$ was proved by Edoukou \cite{E}; followed by the case $d=3$ by Beelen and Datta in \cite{BD} and finally the conjecture for all $d\leq q$ was settled down by Beelen, Datta, and Homma \cite{BDH}.

For $n\geq 4$, initial progress was made by Hallez and Storme, who established the validity of the conjecture for $d=2$ under the additional assumption $n < O(q^2)$ in \cite{HS}. This restriction was later removed by Bartoli, Boeck, Fanali, and Storme, who proved the conjecture for all integers $n\geq 4$ in the case $d=2$ \cite[Theorem 3.3]{BBFS}.

More recently, the case $n=4$ and $d=3$ was established by Datta and the present author in \cite{DM}, under the assumption $q\geq 7$. Furthermore, the conjecture has been proved for $d=3$ and all $n\geq 4$, again under the same hypothesis $q\geq 7$, by the present author in \cite{S1}. For degrees $d\geq 4$, Conjecture~\ref{ELXc} remains open to date. Thus, the preceding discussion leads us to the following: 

\begin{theorem}\label{lit}
Let $n \geq 4$ and let $q$ be a prime power. Then
$$
\begin{aligned}
M_n(1) &=
\begin{cases}
|\U_{n-1}(\FF_{q^2})| 
\ \text{if} \ n \ \text{is even} \\
q^2 |\U_{n-2}(\FF_{q^2})| + 1 
\ \text{if} \ n \ \text{is odd},
\end{cases} \\
M_n(2) &=
\begin{cases}
2|\U_{n-1}(\FF_{q^2})| - |\U_{n-2}(\FF_{q^2})| \ \text{if} \ n \ \text{is even}
 \\
(2q^2 - 1)|\U_{n-2}(\FF_{q^2})| + 2
\ \text{if} \ n \ \text{is odd}
\end{cases} \\ \text{and, for}\ q\geq 7\\
M_n(3) &=
\begin{cases}
3|\U_{n-1}(\FF_{q^2})| - 2|\U_{n-2}(\FF_{q^2})|
\ \text{if} \ n \ \text{is even} \\
(3q^2 - 2)|\U_{n-2}(\FF_{q^2})| + 3
\ \text{if} \ n \ \text{is odd}.
\end{cases}
\end{aligned}
$$ The value of $M_n(d)$ is still unknown for $d \ge 4$ when $n \ge 4$.
\end{theorem}

%%%%%%%%%%%%%%%%%%%%%%%%%%%%%%%%%%%%%%%
\section{ Parameters of the functional codes}\label{parameters}
We now determine the parameters of the functional codes appearing from the forms of degree $d\leq q$ and the rank $n$ Hermitian variety. To this end, we first present the definition of functional codes introduced by Lachaud.
\begin{definition}\cite[Definition 5.4]{L}\label{functional code}
\normalfont Let $\Fq[x_0,\dots,x_n]_d$ denote the $\Fq$-vector space of all homogeneous polynomials in $\Fq[x_0,\dots,x_n]$ of degree $d$ together with the zero polynomial. Suppose $\X$ is a projective variety in $\PP^n(\Fq)$ and $\p_1,\p_2,\dots,\p_{|\X(\Fq)|}$ is an ordered list of all $\Fq$-rational points of $\X$, chosen so that the last non-zero coordinate of each point is $1$. Consider the evaluation map
\begin{equation}\label{ev}
\mathrm{ev}:\Fq[x_0,\dots,x_n]_d \to \Fq^{|\X(\Fq)|} \ \ \text{defined by} \ \ F\mapsto c_F:=\left(F\left(\p_1\right),\dots,F\left(\p_{|\X(\Fq)|}\right)\right).
\end{equation}
The map $\mathrm{ev}$ is a linear map and its image $\mathrm{Im}\left(\mathrm{ev}\right)$ is called the \textit{functional code $C_d(\X)$ defined by forms of degree $d$ on $\X$}.
\end{definition} 

\begin{remark}\label{re: parameter}
 \normalfont We observe that the \textit{length} of the code $C_d(\X)$ is $|\X(\Fq)|$ and the \textit{dimension} is  $\binom{n+d}{d}-\dim_{\Fq}\left(\mathrm{ker}\left(\mathrm{ev}\right)\right)$, where $\mathrm{ev}$ is defined in \eqref{ev}. The \textit{minimum distance} of $C_d(\X)$ is given by
 \begin{align*}
     \mathsf{d}(C_d(\X)) &=\min\left\{\wt(c_F): c_F\in C_d(\X),\ c_F\neq 0\right\}\\
                        &=|\X(\Fq)|-\max\left\{|V(F)\cap \X(\Fq)|: c_F\neq 0\right\}.
\end{align*}
 Note that, if the evaluation map $\mathrm{ev}$ is \textit{injective} then $c_F\neq 0$ if and only if $F\neq 0$. In that case, we have 
 $$\mathsf{d}(C_d(\X))=|\X(\Fq)|-\max\left\{|V(F)\cap \X(\Fq)|: F\neq 0\right\}.$$
\end{remark}

%%%%%%%%%%%%%%%%%%%%%%%%%%%%%%%%%%%%
\begin{lemma}\label{inj}
  Let $\X=P\U_{n-1}$ be the degenerate Hermitian variety of rank $n$ in $\PP^n(\Fqt)$ with $n\geq 3$. If $d\leq q$, then the evaluation map 
\begin{align*}
  \ev: \Fqt[x_0,\dots,x_n]_d &\to \Fqt^{|P\U_{n-1}(\Fqt)|} \\
  F&\mapsto c_F:=\left(F\left(\p_1\right),\dots,F\left(\p_{|P\U_{n-1}(\Fqt)|}\right)\right)
\end{align*}
  is injective.
\end{lemma}
\begin{proof}
    The proof is distinguished into three cases.

    \textbf{Case 1:} $n\geq 4$.
It follows from Serre's inequality \ref{serre} that if $F\in\Fqt[x_0,\dots,x_n]$ is a non-zero homogeneous polynomial of degree $d$, then 
\begin{align*}
    |V(F)(\Fqt)|&\leq dq^{2n-2}+\pi_{n-2}(\Fqt)\\
                &\leq q^{2n-1}+\pi_{n-2}(\Fqt),        
\end{align*}
where the last inequality follows since $d\leq q$. Now, since $\pi_{n-2}(\Fqt)=(q^{2n-2}-1)/(q^2-1)$ a straightforward computation yields 
\begin{align*}
    |P\U_{n-1}(\Fqt)|- \left( q^{2n-1}+\pi_{n-2}(\Fqt)\right)&= \frac{q^{n+1}\left(q^{n-3}+(-1)^n\right)}{q+1}\\
                & > 0,
\end{align*}
where the last inequality follows since $n\geq 4$.

\textbf{Case 2:} $n=3$. It follows from Theorem \ref{surface and rank n} that for a non-zero homogeneous polynomial $F$ of degree $d\leq q$, we have
\begin{align*}
  |V(F)\cap P\U_2(\Fqt)|&\leq 1+q^2d(q+1)\\
                        &\leq q^4+q^3+1 \ \ (\text{as}\ d\leq q)\\
                        &< |P\U_2(\Fqt)|=q^5+q^2+1.
\end{align*}

\textbf{Case 3:} $n=2$. In the same vein as Case 2, it follows from Theorem \ref{2-main-theorem} that $|V(F)\cap P\U_1(\Fqt)|\leq 1+dq^2<|P\U_1|=q^3+q^2+1$.

Combining the cases, we conclude that a homogeneous polynomial $F\in \Fqt[x_0,\dots,x_n]$ of degree $d\leq q$ never vanishes on $P\U_{n-1}$ for $d\leq q$, unless it is identically zero. 
\end{proof}

\begin{theorem}\label{th:parameters}
    The code $C_d(P\U_{n-1})$ defined by forms of degree $d\leq q$ on $P\U_{n-1}$ is an $[m,k,\mathsf{d}]_{q^2}$ code where 
    \begin{align*}
        &m=|P\U_{n-1}(\Fqt)|=1+q^2\frac{(q^{n-1}-(-1)^{n-1})(q^n-(-1)^n)}{q^2-1}\\
        &k=\binom{n+d}{d}\\
        & \mathsf{d}\geq |P\U_{n-1}|- \max \left\{|\U_{n-1}(\Fqt)|+(d-1)(q+1)q^{2n-4}, \ 1+q^2 M_{n-1}(d) \right\}. 
    \end{align*}
\end{theorem}
\begin{proof}
  The length and dimension follow from remark \ref{re: parameter}, whereas the minimum distance follows from Theorem \ref{rank n} and Lemma \ref{inj}.  
\end{proof}

We now present the parameters of the code $C_d(P\U_{n-1})$ in the case when $n=2,3,4$.
\begin{theorem}\label{codes PU}
Let $d\le q$. The code $C_d(P\U_{n-1})$ defined by forms of degree $d$ on $P\U_{n-1}$ is an $[m,k,\mathsf{d}]_{q^2}$-code, where if $n=2$, then
\begin{align*}
m &= |P\U_{1}(\Fqt)| = q^3 + q^2 + 1\\
k &= \binom{2+d}{d}\\
\mathsf{d} &= q^3-(d-1)q^2,
\end{align*}

 if $n=3$, then
\begin{align*}
m &= |P\U_{2}(\Fqt)| = q^5 + q^2 + 1\\
k &= \binom{3+d}{d}\\
\mathsf{d} &= q^2\left(q^3 - dq - (d-1)\right),
\end{align*}

and if $n=4$, then
\begin{align*}
m &= |P\U_{3}(\Fqt)| = q^7 + q^5 + q^4 + q^2 + 1\\
k &= \binom{4+d}{d}\\
\mathsf{d} &= q^7 - (d-1)q^3(q^2+q-1).
\end{align*}
\end{theorem}

\begin{proof} 
The case $n=2$ follows from Theorem \ref{2-main-theorem} and Theorem \ref{th:parameters}.
The case $n=3$ follows from Theorem \ref{th:parameters} together with Theorem \ref{surface and rank n} and the fact that there exists an $\Fqt$-surface sharing exactly $1+q^2d(q+1)$ many $\Fqt$-rational points with $P\U_2$. 
The case $n=4$ follows from Theorem \ref{th:parameters} together with Theorem \ref{degenerate and threefold} and Remark \ref{when attained}.
\end{proof}

\section{Acknowledgment}
The author would like to thank Dr. Mrinmoy Datta for his encouragement and some helpful suggestions during this work.

%%%%%%%%%%%%%%%%%%%%%%%%%%%%%%%%%%%%%%%%%%%%%%%%%%%%%%%%%%%%%%%%%%%%%%%%%%%%%%%%%%%%%%%%%%%%%%%%%%%%%%%%%%%%%%%%%%%%%%%%%%%%%%%%%%%%%%%%%%%%%%%%%%%%%%

\end{document}